\documentclass[12pt,leqno,a4paper]{article}
\usepackage{graphicx}
\usepackage{caption}
\usepackage{subcaption}
\usepackage{color}

\usepackage[pdfborder={0 0 0}]{hyperref}

\usepackage{amssymb}
\usepackage{amsmath}
\usepackage{amsthm}

\usepackage[shortlabels]{enumitem}

\newtheorem{Thm}{Theorem}[section]

\newtheorem{Lem}[Thm]{Lemma}

\theoremstyle{remark}

\theoremstyle{definition}
\newtheorem{Rem}[Thm]{Remark}
\newtheorem{Def}[Thm]{Definition}
\newtheorem{Ex}[Thm]{Example}

\numberwithin{equation}{section}

\newcommand\set[1]{\left\{\,#1\,\right\}}		
\newcommand\abs[1]{\left|#1\right|}				
\newcommand\ska[1]{\left\langle#1\right\rangle} 
\newcommand\norm[1]{\left\Vert#1\right\Vert}	

\DeclareMathOperator{\range}{range}				
\DeclareMathOperator{\dist}{dist}				
\DeclareMathOperator{\id}{id}					
\DeclareMathOperator{\spann}{span}				
\DeclareMathOperator{\diag}{diag}				
\DeclareMathOperator{\ord}{ord}					

\def\N{\mathbb{N}}
\def\Z{\mathbb{Z}}

\def\R{\mathbb{R}}

\def\T{\mathbb{T}}

\newcommand{\cC}{{\mathcal C}}

\newcommand{\cF}{{\mathcal F}}

\newcommand{\cH}{{\mathcal H}}

\newcommand{\cL}{{\mathcal L}}
\newcommand{\cM}{{\mathcal M}}

\newcommand{\cO}{{\mathcal O}}

\newcommand{\cU}{{\mathcal U}}

\newcommand{\Hpla}{H_{\R^2}}						

\usepackage[backend=biber,style=numeric, url=false, doi=false, isbn=false, giveninits=true, date=year, maxbibnames=99]{biblatex} 
\DeclareNameAlias{author}{first-last}					
\renewbibmacro{in:}{}									
\DeclareFieldFormat[article,inbook,incollection,inproceedings,patent,thesis,unpublished]{title}{\textit{#1}}	
\DeclareFieldFormat[article,inbook,incollection,inproceedings,patent,thesis,unpublished]{journaltitle}{#1}
\DeclareFieldFormat{pages}{#1}							
\DeclareFieldFormat{titlecase}{#1}						
\DeclareFieldFormat[article]{volume}{\mkbibbold{#1}} 	
\DeclareFieldFormat[article]{number}{\mkbibbold{#1}} 	
\DeclareFieldFormat[book]{pagetotal}{}
\begin{document}

\title{Periodic solutions for the N-vortex problem via a superposition principle}
\author{Bj\"orn Gebhard}
\date{}
\maketitle

\begin{abstract}
We examine the $N$-vortex problem on general domains $\Omega\subset\mathbb{R}^2$ concerning the existence of nonstationary collision-free periodic solutions. The problem in question is a first order Hamiltonian system of the form 
$$
\Gamma_k\dot{z}_k=J\nabla_{z_k}H(z_1,\ldots,z_N),\quad k=1,\ldots,N,
$$
where $\Gamma_k\in\R\setminus\{0\}$ is the strength of the $k$th vortex at position $z_k(t)\in\Omega$, $J\in\mathbb{R}^{2\times 2}$ is the standard symplectic matrix and 
$$
H(z_1,\ldots,z_N)=-\frac{1}{2\pi}\sum_{\underset{k\neq j}{k,j=1}}^N\Gamma_j\Gamma_k\log|z_k-z_j|-\sum_{k,j=1}^N\Gamma_j\Gamma_k g(z_k,z_j)
$$
with some regular and symmetric, but in general not explicitely known function $g:\Omega\times\Omega\rightarrow \mathbb{R}$. The investigation relies on the idea to superpose a stationary solution of a system of less than $N$ vortices and several clusters of vortices that are close to rigidly rotating configurations of the whole-plane system. We establish general conditions on both, the stationary solution and the configurations, under which multiple $T$-periodic solutions are shown to exist for every $T>0$ small enough. The crucial condition holds in generic bounded domains and is explicitely verified for an example in the unit disc $\Omega=B_1(0)$. In particular we therefore obtain various examples of periodic solutions in $B_1(0)$ that are not rigidly rotating configurations.
\end{abstract}



\section{Introduction and statement of results}\label{sec:intro}
The $N$-vortex problem is a first order Hamiltonian system that describes the motion of $N$ point vortices inside a planar domain $\Omega\subset \R^2$. If $z_k(t)\in\Omega$ denotes the position of the $k$th vortex at time $t$ and $\Gamma_k\in\R\setminus\{0\}$ its strength, the system is given by
\begin{equation}\label{eq:Ham_omega}
\Gamma_k \dot{z}_k=J\nabla_{z_k} H_\Omega(z_1,\ldots,z_N),\quad k=1,\ldots,N,
\end{equation}
where $J\in\R^{2\times 2}$ is the rotation by $-\frac{\pi}{2}$ and the Hamiltonian $H_\Omega$ defined on
$$
\cF_N(\Omega)=\set{(z_1,\ldots,z_N)\in\Omega^N:z_j\neq z_k\text{ for }j\neq k}
$$ reads
$$
H_\Omega(z_1,\ldots,z_N)=-\frac{1}{2\pi}\sum_{\underset{k\neq j}{k,j=1}}^N\Gamma_j\Gamma_k\log\abs{z_k-z_j}-\sum_{k,j=1}^N\Gamma_j\Gamma_k g_\Omega(z_k,z_j).
$$
The function $g_\Omega:\Omega\times\Omega\rightarrow\R$ classically is defined by the requirement that 
$$
G_\Omega(x,y)=-\frac{1}{2\pi}\log\abs{x-y}-g_\Omega(x,y) 
$$
is the Green's function of the Dirichlet Laplacian of $\Omega$ -- or a more general hydrodynamic Green's function -- and thus in almost all cases not explicitely known.

One obtains system \eqref{eq:Ham_omega} with a point vortex ansatz for the 2D Euler equations, see e.g. \cite{flucher_variational_1999,marchioro_mathematical_1994,newton_n-vortex_2001,saffman_vortex_1993}. Depending on the considered case the derivation originally is due to Kirchhoff \cite{kirchhoff_vorlesungen_1876}, Routh \cite{routh_applications_1880} and Lin \cite{lin_motion_1941,lin_motion_1941-1}. The definition and some properties of hydrodynamic Green's functions can be found in \cite{flucher_variational_1999,flucher_vortex_1997} .

Similar Hamiltonian systems, in which $g_\Omega$ in the definition of $H_\Omega$ is replaced by a possibly different regular function, also appear in singular limits of other PDEs like the Ginzburg-Landau-Schr\"odinger (or Gross-Pitaevskii) equation and the Landau-Lifshitz-Gilbert equation, see \cite{jerrard_dynamics_1998,kurzke_ginzburg-landau_2011} and references therein. In fact for our result it is enough that $g:\Omega\times\Omega\rightarrow\R$ is a sufficiently smooth and symmetric function and not necessarily the regular part of the Dirichlet or a hydrodynamic Green's function.

The present paper will address the question of existence of periodic solutions of \eqref{eq:Ham_omega} in an arbitrary domain. In special domains like $\Omega=\R^2$, $\Omega=B_1(0)$ quite a lot of periodic solutions of \eqref{eq:Ham_omega} can be found that rotate as a fixed configuration around a certain point, cf. section \ref{subsec:relative_equilibria}. This is possible because in those cases $g_\Omega$ is explicitely known and invariant with respect to rotations. Besides the fact that the Hamiltonian is in almost all other cases not explicitely known, it is in general unbounded from both sides, not integrable, has singularities and non compact, not metrically complete energy surfaces. These difficulties cause the failure of standard theorems and methods for the existence of periodics. 

However in the past years three types of periodic solutions in almost arbitrary domains could be established. In the first one vortices with possibly different strengths and of arbitrary number are close to a critical point of the so called Robin function $h_\Omega(z)=g_\Omega(z,z)$ and the configuration of vortices looks after rescaling like a rigidly rotating solution of the $N$-vortex system on $\R^2$, see \cite{bartsch_periodic_2016-1,bartsch_global_2016}. In the second type of solutions, shown in \cite{bartsch_periodic_2016-2}, two identical vortices rotate around their center of vorticity while the center itself follows a level line of $h_\Omega$. The third result holds for an arbitrary number of identical vortices, which separated by time shifts follow the same curve close to the boundary of a simply connected bounded domain, \cite{bartsch_periodic_2016}.  The first and the second result can be seen as a superposition of a solution of a $1$-vortex system in the domain and a solution of the $N$-vortex, resp. $2$-vortex problem on the whole plane. Note here that in the case of a single vortex the Hamiltonian $H_\Omega$ is up to a factor given by the Robin function $h_\Omega$, so critical points of $h_\Omega$ are stationary solutions of the $1$-vortex problem and level lines of $h_\Omega$ correspond to periodic solutions of it. 

Here we will generalize the results of \cite{bartsch_periodic_2016-1,bartsch_global_2016} in the following way:
Instead of an equilibrium of the 1-vortex system on $\Omega$, we consider a stationary solution of a system of $m$-vortices with strength $\Gamma^1,\ldots,\Gamma^m$ located at $\alpha^1,\ldots,\alpha^m\in\Omega$. For every vortex $\Gamma^k$, $k=1,\ldots,m$ take now a rigidly rotating configuration $Z^k(t)$ of the whole-plane system consisting of $N_k$ vortices with strengths $\Gamma^k_1,\ldots,\Gamma^k_{N_k}$, such that $\sum_{j=1}^{N_k}\Gamma^k_j\neq 0$. In the case $N_k=1$ a stationary single vortex may also be considered as an admissible configuration. By a change of timescale we may assume that $\sum_{j=1}^{N_k}\Gamma^k_j=\Gamma^k$. We then ask for the existence of periodic solutions of the $\left(\sum_{k=1}^mN_k\right)$-vortex system on $\Omega$, in which the vortices form $m$ clusters $(z^k_1,\ldots,z^k_{N_k})$, $k=1,\ldots,m$ approximately satisfying 
$$
z^k_j(t)\approx\alpha^k+rZ^k_j(t/r^2)
$$ 
with a small parameter $r>0$. So we superpose a stationary solution of the $m$-vortex system on $\Omega$ and several rigidly rotating configurations of the whole-plane system. This is illustrated for a simple case in Figure \ref{plot1}. 

\begin{figure}
\center
\begin{subfigure}[b]{0.25\textwidth}
        \raisebox{-.47\height}{\includegraphics[width=\textwidth]{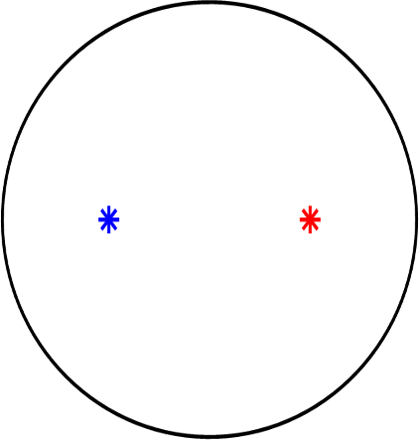}}
    \end{subfigure}
    $\quad+\quad$
    \begin{subfigure}[b]{0.07\textwidth}
        \raisebox{-.42\height}{\includegraphics[width=\textwidth]{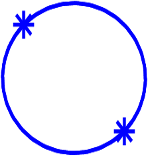}}
    \end{subfigure}
    $\quad+\quad$
    \begin{subfigure}[b]{0.07\textwidth}
        \raisebox{-.42\height}{\includegraphics[width=\textwidth]{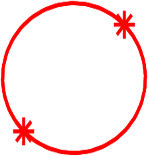}}
    \end{subfigure}
    $\quad=\quad$
    \begin{subfigure}[b]{0.25\textwidth}
        \raisebox{-.47\height}{\includegraphics[width=\textwidth]{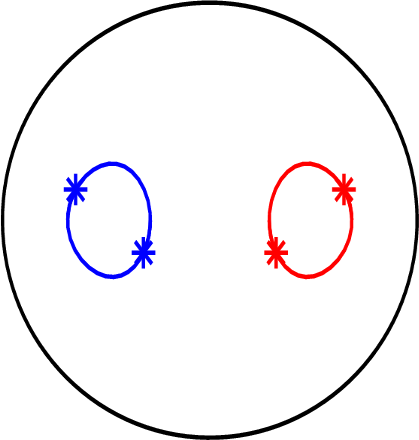}}
    \end{subfigure}
\caption{\small This diagram illustrates the superposition idea. The $2$-vortex problem in the unit disc admits a stationary solution with $\Gamma^1=-\Gamma^2$, cf. Example \ref{Ex:unit_disc}, say $\Gamma^1=-2$ (blue star), $\Gamma^2=2$ (red star). As rigidly rotating configurations on $\R^2$ we take here for simplicity two identical vortices for $\Gamma^1$ and $\Gamma^2$, i.e. $\Gamma^1_1=\Gamma^1_2=-1$ rotate on the blue circle in clockwise direction and $\Gamma^2_1=\Gamma^2_2=1$ rotate on the red circle in counterclockwise direction. The result on the right-hand side is a periodic solution of the $4$-vortex system in the disc with vorticities $\Gamma^1_1,\Gamma^1_2,\Gamma^2_1,\Gamma^2_2$, where each pair of vortices moves along a deformed circle in the same orientation as before. The shown trajectory is the actual numerically computed trajectory of the $4$-vortex problem. Suitable initial conditions can in this case be found due to symmetry considerations.}
\label{plot1}
\end{figure}

The general idea of grouping vortices into different clusters plays a role in establishing the existence of quasi-periodic solutions via KAM theory, see \cite{khanin_quasi-periodic_1982, marchioro_vortex_1984}. 
In this paper we use it to provide general conditions that give rise to families of periodic solutions. The conditions will be verified for a concrete case in the unit disc $\Omega=B_1(0)$ leading to examples of periodic solutions with an arbitrary number of $N\geq 3$ vortices that are not rigidly rotating configurations, one of them is presented in Figure \ref{plot1}.

In the following subsections we will formulate two versions of our theorem and discuss how far the conditions of the theorems hold. Details on the needed ingredients, i.e. stationary solutions of a $m$-vortex system in $\Omega$ and rigidly rotating solutions of the whole-plane system, together with required properties are given in sections \ref{subsec:critical_points_of_cHOmega} and \ref{subsec:relative_equilibria}. After that in section \ref{sec:ansatz} we set up an equation on a Hilbert space that we have to solve in order to get the desired periodic solutions. The equation depends on a parameter $r>0$ which is introduced through a rescaling of the problem. The main part of the  proof of Theorem \ref{thm:2} in section \ref{sec:proof_of_thm2} is to overcome natural degeneracies of the limiting equation when $r\rightarrow 0$. Section \ref{sec:the_case_l_1} contains additional information for a special case of Theorem \ref{thm:2}. Finally we verify in section \ref{sec:example} the conditions for a concrete example in the unit disc.

Bevor we state our results we shortly like to mention the conclusions one can draw from solutions of the $N$-vortex system for the PDEs that give rise to this system as some sort of singular limit. By constructing appropriate stream functions it is possible to desingularize stationary solutions of the $N$-vortex problem to stationary solutions of the 2D Euler equations, see \cite{cao_regularization_2014} and references therein. 
A similar result for the Euler equations and periodic solutions is so far not available. Concerning other PDEs Venkatraman has shown in \cite{venkatraman_periodic_2017} that rigidly rotating solutions of \eqref{eq:Ham_omega} in the unit disc give rise to corresponding periodic solutions of the Gross-Pitaevskii equation. The same is true for rigidly rotating configurations on the sphere $S^2$, see \cite{gelantalis_rotating_2012}. Apart from that the desingularization of general periodic solutions like the ones obtained here is also for the Gross-Pitaevskii equation an open problem. 

\subsection{Statement of results part 1}\label{sec:results1}
Let $\Omega\subset\R^2$ be a domain and fix a symmetric $\cC^2$ function $g:\Omega\times\Omega\rightarrow \R$, for example the regular part of a hydrodynamic Green's function of $\Omega$. We will investigate a point vortex like system similar to \eqref{eq:Ham_omega}, which is induced by the generalized Green's and Robin functions
$$
G(x,y)=-\frac{1}{2\pi}\log\abs{x-y}-g(x,y),\quad\quad h(x)=g(x,x).
$$

At first we consider on the domain $\Omega$ a system of $m\in\N$ vortices with  vorticities $\Gamma^1,\ldots,\Gamma^m\in\R\setminus\{0\}$ and Hamiltonian
$$
\cH(a)=\sum_{\underset{k\neq k'}{k,k'=1}}^m\Gamma^k\Gamma^{k'}G(a^k,a^{k'})-\sum_{k=1}^m\Gamma^k\Gamma^kh(a^k)
$$
defined on $\cF_m(\Omega)=\set{a=(a^1,\ldots,a^m)\in\Omega^m:a^k\neq a^{k'}\text{ for all }k\neq k'}$. We require that the corresponding $m$-vortex system admits a stationary solution, cf. section \ref{subsec:critical_points_of_cHOmega}. To be more precise we assume
\begin{enumerate}[leftmargin=2.5\parindent]
\item[(A1)] $\cH$ has a nondegenerate critical point $\alpha\in\cF_m(\Omega)$.
\end{enumerate}

Next we fix a number $l\in\set{1,\ldots,m}$, which will be the number of vortices that are splitted into configurations consisting of more than a single vortex. Without restriction we take the first $l$ vortices. I.e. for $k=1,\ldots,l$ choose $N_k\geq 2$ vorticities $\Gamma^k_1,\ldots,\Gamma^k_{N_k}\in\R\setminus\{0\}$, such that 
\begin{enumerate}[leftmargin=2.5\parindent]
\item[(A2)] $\sum_{j=1}^{N_k}\Gamma_j^k=\Gamma^k$.
\end{enumerate}
We then define the Hamiltonian $H^k_{\R^2}:\cF_{N_k}(\R^2)\rightarrow\R$,
$$
\Hpla^k(z)=-\frac{1}{2\pi}\sum_{\underset{j\neq j'}{j,j'=1}}^{N_k}\Gamma_j^k\Gamma_{j'}^k\log\abs{z_j-z_{j'}}
$$
inducing the $N_k$-vortex system 
\begin{equation}\label{eq:Ham_N_k_vortex_system}
\Gamma^k_j\dot{z}_j=J\nabla_{z_j} \Hpla^k(z),\quad j=1,\ldots,N_k
\end{equation}
on $\R^2$.

As mentioned in the introduction a $\tilde{N}$-vortex system on $\R^2$ allows rigidly rotating solutions, also called relative equilibria, of the form $Z(t)=e^{\omega J_{\tilde{N}}t}z$, $\omega\neq 0$, cf. section \ref{subsec:relative_equilibria} for examples. Here $J_{\tilde{N}}=\diag(J,J,\ldots,J)\in\R^{2\tilde{N}\times2\tilde{N}}$. Due to scaling $Z(t)\rightarrow \lambda Z(t/\lambda^2)$, $\lambda>0$, we can assume $\omega=\pm 1$. The corresponding $2\pi$-periodic relative equilibrium is called nondegenerate, if the linearized equation
\begin{equation}\label{eq:linearization_on_R2}
\Gamma_j\dot{w}_j=J(\nabla^2\Hpla(Z(t))w)_j,\quad j=1,\ldots,\tilde{N}
\end{equation}
has only 3 linear independent $2\pi$-periodic solutions. This is the minimal possible number due to the invariance under rotations and translations. Our third requirement is:
\begin{enumerate}[leftmargin=2.5\parindent]
\item[(A3)] For $k\in\set{1,\ldots,l}$ there exists a $2\pi$-periodic nondegenerate relative equilibrium solution $Z^k(t)=e^{\pm J_{N_k}t}z^k$ of \eqref{eq:Ham_N_k_vortex_system}.
\end{enumerate}

Note that condition (A2) can always be achieved by a change of time scale provided one has a relative equilibrium solution of \eqref{eq:Ham_N_k_vortex_system} with $\sum_j\Gamma^k_j\neq 0$. 

The remaining $m-l$ vortices --  which may be none -- are not splitted into configurations. I.e. for $k=l+1,\ldots,m$ we let $N_k=1$, $\Gamma^k_1=\Gamma^k$, $H^k_{\R^2}:\R^2\rightarrow\R$, $H^k_{\R^2}\equiv 0$ and $Z^k:\R\rightarrow\R^2$, $Z^k(t)\equiv 0$. 

The system under investigation is the generalized $N:=\sum_{k=1}^{m}N_k$-vortex system
\begin{equation}\label{eq:general_ham_system}
\Gamma_j^k\dot{z}_j^k=J\nabla_{z^k_j}H(z),\quad k=1,\ldots,m,~j=1,\ldots,N_k,
\end{equation}
with Hamiltonian 
$$
H(z)=\sum_{(k,j)\neq(k',j')}\Gamma^k_j\Gamma^{k'}_{j'}G(z^k_j,z^{k'}_{j'})-\sum_{(k,j)}\Gamma^k_j\Gamma^k_jh(z^k_j).
$$
Here $z=(z^1_1,\ldots,z^1_{N_1},\ldots,z^m_1,\ldots,z^m_{N_m})\in\cF_N(\Omega)$ and the indices of the sums run through $\set{(k,j):1\leq k\leq m, 1\leq j\leq N_k}$.
We equivalently write for \eqref{eq:general_ham_system}
$$
M_\Gamma\dot{z}=J_N\nabla H(z)
$$
with $M_\Gamma=\diag\big(\Gamma^1_1,\Gamma^1_1,\ldots,\Gamma^1_{N_1},\Gamma^1_{N_1},\ldots,\Gamma^m_1,\Gamma^m_1,\ldots,\Gamma^m_{N_m},\Gamma^m_{N_m}\big)\in \R^{2N\times 2N}$ and $J_N=\diag\big(J,\ldots,J\big)\in\R^{2N\times 2N}$.

We will use the Sobolev spaces $H^1_{T}=H^1(\R/T\Z,\R^{2N})$, $T>0$ of continuous $T$-periodic functions with square-integrable derivative, equipped with the scalar product
$$
\ska{u,v}_{H^1_T}=\int_0^T\ska{u,v}_{\R^{2N}}\:dt+\int_0^T\ska{\dot{u},\dot{v}}_{\R^{2N}}\:dt
$$
and induced norm $\norm{\cdot}_{H^1_T}$.
For $Z^1,\ldots,Z^m$ as defined before let 
\begin{equation}\label{eq:definition_of_M}
\cM=\set{\left(Z^1(\cdot+\theta_1),\ldots,Z^m(\cdot+\theta_m)\right):\theta_1,\ldots,\theta_m\in\R}\subset H^1_{2\pi},
\end{equation}
which is a $l$-dimensional submanifold, since $Z^{l+1}=\ldots=Z^{m}=0$.
And for $a=(a^1,\ldots,a^m)\in\R^{2m}$ we define
$$
\hat{a}=(a^1,\ldots,a^1,a^2,\ldots,a^2,\ldots,a^m,\ldots,a^m)\in \R^{2N_1}\times\ldots\times\R^{2N_m}=\R^{2N}.
$$
Now we are ready to formulate a first version of our theorem.
\begin{Thm}\label{thm:1}
Under the assumptions (A1)-(A3) there exists $T_0>0$ such that for each $T\in(0,T_0)$ the $N$-vortex type system \eqref{eq:general_ham_system} has $l$ distinct $T$-periodic solutions that are in the following sense close to $\alpha$ and $(Z^1,\ldots,Z^m)$: Let $(z_n)_{n\in\N}$ be a sequence consisting of these periodic solutions with periods $T_n\rightarrow 0$ as $n\rightarrow\infty$, then the $k$th components $[z_n]^k_j$, $j=1,\ldots,N_k$ converge to $\alpha^k$ as $n\rightarrow\infty$, $k=1,\ldots,m$. 
Moreover if we rescale $z_n$, such that 
$$
z_n(t)=r_nu_n\left(\frac{t}{r_n^2}\right)+\hat{\alpha},\quad r_n=\sqrt{\frac{T_n}{2\pi}},\quad u_n\in H^1_{2\pi}, 
$$
then $\dist(u_n,\cM)\rightarrow 0$ with respect to $\norm{\cdot}_{H^1_{2\pi}}$ as $n\rightarrow\infty$.
\end{Thm}
So roughly speaking we can split vortices of a stationary solution into suitable rigidly rotating configurations and obtain periodic solutions. For fixed $T\in(0,T_0)$ the multiplicity of the $T$-periodic solutions is based on the relative orientation of the $l$ nontrivial configurations to each other. 
Note that the conditions (A1), (A3) are only related to each other in the sense that the vorticities need to add up as stated in (A2). Also the specific relative equilibrium solutions can be choosen independently of each other. Under an additional technical assumption, that couples the critical point of $\cH$ and the relative equilibria $Z^k$, one can improve the multiplicity from $l$ to $2^{l-1}$ $T$-periodic solutions, cf. Remark \ref{rem:A4}. 

Next we will discuss and improve assumptions (A1), (A3) with respect to their applicability to the classical $N$-vortex system \eqref{eq:Ham_omega}. Whenever we provide a function with an index $\Omega$, like $\cH_\Omega$, we refer to the corresponding function induced by the regular part of the Dirichlet Green's function.  
\subsection{\texorpdfstring{Critical points of $\cH_\Omega$}{About critical points of the Hamiltonian}}\label{subsec:critical_points_of_cHOmega}
The search for stationary solutions in general domains itself is not an easy task. 
Of course there is one trivial case: If $m=1$ the $1$-vortex Hamiltonian $\cH_\Omega$ coincides up to a factor with the Robin function $h_\Omega$, which always has a Minimum in bounded domains. 

Concerning more vortices only in the last years some results on the existence of critical points of the $N$-vortex -- in our case $m$-vortex -- Hamiltonian for bounded domains could be achieved, examples include:
\begin{itemize}
\item $m\in\N$, $\Gamma^1=\ldots=\Gamma^m\neq 0$ and $\Omega$ not simply connected \cite{del_pino_singular_2005} or dumbell shaped \cite{esposito_existence_2005},
\item $m\in\set{2,3,4}$, conditions on $\Gamma^k$, e.g. $m=2$ and $\Gamma^1\Gamma^2<0$, $\Omega$ arbitrary \cite{bartsch_critical_2015},
\item $m\in\N$, conditions on $\Gamma^k$ (different from the ones in \cite{bartsch_critical_2015}) for $\Omega$ arbitrary and for $\Omega$ not simply connected \cite{kuhl_equilibria_2016},
\item $m\in\N$, $\Gamma^k=(-1)^{k+1}\Gamma^1$, $\Omega$ symmetric with respect to reflection at a line \cite{bartsch_n-vortex_2010} or the action of a dihedral group \cite{kuhl_symmetric_2015}. 
\end{itemize}
None of the mentioned results addresses the question of nondegeneracy of the critical points, on which our proof relies. Indeed condition (A1) is for these solutions hard to check, since the Hamiltonian $\cH_\Omega$ and the critical point $\alpha$ are not explicitely known. However a recent result of Bartsch, Micheletti and Pistoia shows that $\cH_\Omega$ has only nondegenerate critical points for a generic bounded domain $\Omega$, see \cite{bartsch_morse_????}. So if the vorticities $\Gamma^1,\ldots,\Gamma^m$ allow the existence of a critical point of $\cH_\Omega$, as for example in one of the listed cases, then condition (A1) is satisfied at least after an arbitrarily small deformation of the domain. 
   
In some cases also explicit stationary configurations are known, for example if $\Omega=\R^2$ or $\Omega=B_1(0)$. But these are all degenerate due to the symmetries of the domain, i.e. if $\alpha\in\cF_m(B_1(0))$ is a critical point of $\cH_{B_1(0)}$, then every $e^{\lambda J_m}\alpha$, $\lambda\in\R$ is a critical point as well. Thus $J_m\alpha\in\ker\nabla^2\cH_{B_1(0)}(\alpha)$ and condition (A1) is violated. But we will see that degeneracy induced by symmetries can still be handled, i.e. we may replace assumption (A1) by
\begin{enumerate}[leftmargin=2.5\parindent]
\item[(A1$^\prime$)] $\cH$ has a critical point $\alpha\in\cF_m(\Omega)$ and one of the following properties holds:
\vspace*{-8pt}
\begin{enumerate}[(i)]
\item $\alpha$ is nondegenerate, 
\item $\Omega$ and $g$ are radial \big($e^{\lambda J}\Omega=\Omega$, $g\left(e^{\lambda J}x,e^{\lambda J}y\right)=g(x,y)$ for every $\lambda\in\R$, $x,y\in\Omega$\big) and  $\dim\ker\nabla^2\cH(\alpha)=1$,
\item $\Omega$ and $g$ are in one direction translational invariant \big(there exists $\nu\in\R^2\setminus\{0\}$ with $\lambda\nu+\Omega=\Omega$, $g(x+\lambda\nu,y+\lambda\nu)=g(x,y)$ for every $\lambda\in\R$, $x,y\in\Omega$\big) and  $\dim\ker\nabla^2\cH(\alpha)=1$,
\item $\Omega=\R^2$, $g(x,y)=\tilde{g}(\abs{x-y})$ and $\dim\ker\nabla^2\cH(\alpha)=3$.
\end{enumerate}
\end{enumerate}
Note that in the classical case $g=g_\Omega$ always inherits the symmetries of the domain.
\begin{Ex}\label{Ex:unit_disc}
Let $\Omega$ be the unit disc $B_1(0)$ and  $g=g_{B_1(0)}$ be the regular  part of the Dirichlet Green's function of $B_1(0)$, which is given by
$$
g(x,y)=g_{B_1(0)}(x,y)=-\frac{1}{4\pi}\log\left(\abs{x}^2\abs{y}^2-2\ska{x,y}_{\R^2}+1\right).
$$ 
The $2$-vortex Hamiltonian $\cH_{B_1(0)}$ with vorticities $\Gamma^1=1$, $\Gamma^2=-1$ satisfies (A1$^\prime$) with a degenerate critical point $\alpha=\big((\mu,0),(-\mu,0)\big)$, where $\mu=\sqrt{\sqrt{5}-2}$. This will be shown in section \ref{sec:example}. 
\end{Ex}

\begin{Rem} If $\Omega=\R^2$, $g=g_{\R^2}\equiv 0$ then critical points of $\cH_{\R^2}$ exist depending on the vorticities $\Gamma^1,\ldots,\Gamma^m$. In the easiest case $m=3$ vortices with strengths $\Gamma^k$ satisfying $\Gamma^1\Gamma^2+\Gamma^1\Gamma^3+\Gamma^2\Gamma^3=0$ are stationary when placed at certain distances along a fixed line, see Theorem 2.2.1 in \cite{newton_n-vortex_2001}. More on stationary configurations can also be found in \cite{aref_vortex_2003}. However for every critical point $\alpha$ of $\cH_{\R^2}$ the inequality $\dim\ker \nabla^2\cH_{\R^2}(\alpha)\geq 4$ holds true. Here $3$ dimensions of the kernel are induced by translations and rotations of the critical point $\alpha$. A fourth dimension by scaling, since differentiation of $\lambda\mapsto \cH_{\R^2}(\lambda\alpha)$ at $\lambda=1$ shows that $\sum_{k\neq k'}\Gamma^k\Gamma^{k'}=0$ is a necessary condition for the existence of critical points. Therefore we have $\cH_{\R^2}(\lambda\alpha)=\cH_{\R^2}(\alpha)$ and 
$$
\ker\nabla^2\cH_{\R^2}(\alpha)\supset \set{(a,\ldots,a)\in\R^{2m}}\oplus\R J_m\alpha\oplus\R\alpha.
$$
 This means that (A1$^\prime$) never holds for critical points of the classical $m$-vortex Hamiltonian $\cH_{\R^2}$, cf. Remark \ref{rem:A2_violated}.
\end{Rem}
\begin{Rem}
Another idea for the existence of periodic solutions is the application of a Weinstein-Moser Theorem \cite{bartsch_generalization_1997, moser_periodic_1976,weinstein_normal_1973} to obtain periodics for the Hamiltonian $\cH_\Omega$ itself via bifurcation from the critical point $\alpha$. But here one encounters the difficulties that $\alpha$ and $\cH_\Omega$ are not explicitely known as well.
\end{Rem}
\subsection{\texorpdfstring{Relative equilibria on $\R^2$}{About relative equilibria}}
\label{subsec:relative_equilibria}
For the $N$-vortex problem on $\Omega=\R^2$ quite a lot of rigidly rotating vortex configurations are known, see \cite{aref_relative_2011,aref_vortex_2003} for an overview. Checking the nondegeneracy condition of such a configuration is, after writing \eqref{eq:linearization_on_R2} in a rotating coordinate frame, a matter of calculating the spectrum of a $2N\times 2N$ matrix. The spectral properties of this matrix are also of interest in the investigation of the  linear stability of the configuration as a periodic solution. So we can use results of Roberts, \cite{roberts_stability_2013} to verify the nondegeneracy.
\begin{Ex}\label{ex:relative_equilibria} The following relative equilibrium solutions are nondegenerate after normalization (scaling and translation):
\begin{itemize}
\item $N=2$, $\Gamma_1+\Gamma_2\neq 0$, $Z(0)\in\cF_2(\R^2)$ arbitrary, cf. Example 2.3 in \cite{bartsch_global_2016},
\item $N=3$, $\Gamma_1+\Gamma_2+\Gamma_3\neq 0$, $0\neq \Gamma_1\Gamma_2+\Gamma_1\Gamma_3+\Gamma_2\Gamma_3\neq \Gamma_1^2+\Gamma_2^2+\Gamma_3^2$, $Z_1(0),Z_2(0),Z_3(0)$ forming an equilateral triangle, cf. Example 2.4 in \cite{bartsch_global_2016},
\item $N\in\N$, $\Gamma_1=\ldots=\Gamma_N$, $Z_j(0)=(x_k,0)$, $j=1,\ldots,N$ with $x_1,\ldots,x_N$ being the roots of the $N$th Hermitian polynomial, see Corollary 3.3 in \cite{roberts_stability_2013}. 
\end{itemize}
\end{Ex}
Observe that the condition for the equilateral triangle configuration excludes the special case $\Gamma_1=\Gamma_2=\Gamma_3$. Nonetheless with a second refinement we can also treat this case leading to solutions for \eqref{eq:general_ham_system} in which the vortices of a subgroup may form choreographies.

The permutation group $\Sigma_{N}$ of $N$ symbols acts orthogonally on $\R^{2N}$ via permutation of components, i.e.
$$
\sigma*z=\left(z_{\sigma^{-1}(1)},\ldots,z_{\sigma^{-1}(N)}\right),\quad \sigma\in\Sigma_{N},~z\in\R^{2N}.
$$
\begin{Def}A relative equilibrium solution $Z(t)$ of the whole plane system is called $\sigma$-nondegenerate, provided $\sigma*Z(\cdot+2\pi)=Z$ and \eqref{eq:linearization_on_R2} has only three linear independent solutions satisfying $\sigma*w(\cdot+2\pi)=w$. 
\end{Def} 
Note that every nondegenerate relative equilibrium is $\sigma$-nondegenerate with $\sigma=\id_{\Sigma_N}$. 
As a nontrivial example we have
\begin{Ex}\label{ex:relative_equilibrium_2}
$N\in\N$ identical vortices placed at the vertices of a regular $N$-Gon form a rigidly rotating configuration, called Thomson's $N$-Gon configuration. It is (after scaling) a $\sigma$-nondegenerate relative equilibrium solution with $\sigma=(1~2~\ldots~N)\in\Sigma_N$, see Lemma 4.1 in \cite{bartsch_periodic_2016-1}.
\end{Ex}
Concerning our situation we weaken assumption (A3) to
\begin{enumerate}[leftmargin=2.5\parindent]
\item[(A3$^\prime$)] For each $k\in\set{1,\ldots,l}$ there exists $\sigma_k\in\Sigma_{N_k}$ with $\Gamma^k_j=\Gamma^k_{\sigma^{-1}_k(j)}$  for every $j=1,\ldots,N_k$, together with a $\sigma_k$-nondegenerate relative equilibrium solution $Z^k(t)=\exp\big(\pm J_{N_k}t/\ord(\sigma_k)\big)z^k$ of \eqref{eq:Ham_N_k_vortex_system}. For consistency in notation let $Z^k\equiv 0$ and $\sigma_k=\id_{\Sigma_1}$ when $k\in\set{l+1,\ldots,m}$.
\end{enumerate}

\subsection{Statement of results part 2}
For $\left(\sigma_k,Z^k\right)_{k=1}^m$ as in (A3$^\prime$) let $\tau=2\pi\ord(\sigma)$, where $\ord(\sigma)$ denotes the order of $\sigma=(\sigma_1,\ldots,\sigma_m)\in\prod_k\Sigma_{N_k}$, further on let $\sigma*z=(\sigma_1*z^1,\ldots,\sigma_m*z^m)$ for $z=(z^1,\ldots,z^m)\in\R^{2N}$. Observe that $\cM$ as defined in \eqref{eq:definition_of_M} is now contained in $H^1_\tau$.
We have the following generalization of Theorem \ref{thm:1}.
\begin{Thm}\label{thm:2}
Assume that (A1$^{\hspace{1pt}\prime}$), (A2) and (A3$^{\hspace{1pt}\prime}$) hold. Then there exists $T_0>0$ such that \eqref{eq:general_ham_system} has $l$ distinct $T$-periodic orbits for every $T\in(0,T_0)$. 
Similar to Theorem \ref{thm:1} if we rescale a sequence $(z_n)_{n\in\N}$ of these solutions with periods $T_n\rightarrow 0$ by
$$
z_n(t)=r_nu_n\left(\frac{t}{r_n^2}\right)+\hat{\alpha},\quad r_n=\sqrt{\frac{T_n}{\tau}},\quad u_n\in H^1_\tau,
$$
then $\dist(u_n,\cM)\rightarrow 0$ in $H^1_\tau$.
 Additionally  the $k$th subgroup, $k=1,\ldots,m$ of vortices $z^k(t)=(z^k_1(t),\ldots,z^k_{N_k}(t))$ of one of the $T$-periodic solutions $z(t)$ inherits the symmetry of the relative equilibrium $Z^k(t)$, i.e.
$$
\sigma*z(t+T/\ord(\sigma))=z(t).
$$
\end{Thm}
In the case that only the first vortex is splitted up into a configuration with at least two vortices, i.e. when $l=1$, we can slightly improve Theorem \ref{thm:2}.
\begin{Thm}\label{thm:case_m1} Let $l=1$, $g\in\cC^k(\Omega\times\Omega,\R)$ with $k\geq 2$. If (A1$^{\hspace{1pt}\prime}$)-(A3$^{\hspace{1pt}\prime}$) hold, then there exists $r_1>0$ and a $\cC^{k-2}$ map $u:[0,r_1)\rightarrow H^1_\tau$, $r\mapsto u^{(r)}$
with $u^{(0)}=Z=(Z^1,0,\ldots,0)\in H^1_\tau$, $\sigma*u^{(r)}(\cdot+2\pi)=u^{(r)}$ and such that 
$$
z^{(r)}(t)=ru^{(r)}\left(\frac{t}{r^2}\right)+\hat{\alpha}
$$
is a $\tau r^2$-periodic solution of \eqref{eq:general_ham_system} for every $r\in(0,r_1)$. Moreover if $k\geq 3$, then
$$
\partial_r u^{(0)}\in\set{\hat{a}:a\in\R^{2m}}\subset H^1_\tau.
$$
\end{Thm}

\section{Ansatz and preliminaries}\label{sec:ansatz}
Fix $\alpha$, $Z^k$, $\sigma_k$, $k=1,\ldots,m$ according to (A1$^\prime$), (A3$^\prime$) and let $\sigma=(\sigma_1,\ldots,\sigma_m)$.
We are looking for a solution $z:\R\rightarrow\cF_N(\Omega)$ where each subgroup of vortices $(z^k_1(t),\ldots,z^k_{N_k}(t))$ is located near $\alpha^k$ and forms a configuration close to a scaled version of the relative equilibrium $Z^k(t)$.

In order to reformulate the problem we define
\begin{equation*}\label{eq:definition_of_F}
F(z)=\sum_{\underset{k\neq k'}{k,k'=1}}^m\sum_{j=1}^{N_k}\sum_{j'=1}^{N_{k'}}\Gamma^k_j\Gamma^{k'}_{j'}G(z^k_j+\alpha^k,z^{k'}_{j'}+\alpha^{k'})-\sum_{k=1}^m\sum_{j,j'=1}^{N_k}\Gamma^k_j\Gamma^k_{j'}g(z^k_j+\alpha^k,z^k_{j'}+\alpha^k)
\end{equation*}
together with the following Hamiltonians $H_0:\cO_0:=\cF_{N_1}(\R^2)\times\ldots\times\cF_{N_m}(\R^2)\rightarrow \R$,
$$
H_0(u)=\sum_{k=1}^m\Hpla^k(u^k_1,\ldots,u^k_{N_k})
$$
and for $r>0$, $H_r:\cO_r:=\set{u\in\R^{2N}:ru+\hat{\alpha}\in\cF_N(\Omega) }\rightarrow\R$,
$$
H_r(u)=H_0(u)+F(ru)-\cH(\alpha).
$$
Observe that $F$ is defined on an open subset of $\R^{2N}$ containing $0$.
\begin{Lem}\label{Lem:Scaling_idea_for_superposition}
Let $I$ be an open intervall and $r>0$. Then $z(t)=ru(t/r^2)+\hat{\alpha}$ solves \eqref{eq:general_ham_system} on $I$ if and only if $u$ solves 
\begin{equation}\label{eq:Ham_r}
M_\Gamma\dot{u}=J_N\nabla H_r(u)
\end{equation}
on $r^2I$.
\end{Lem}
\begin{proof}
Clearly $z(t)$ as above is a solution of \eqref{eq:Ham_omega} if and only if 
$$
M_\Gamma\dot{u}=rJ_N\nabla H(ru+\hat{\alpha})
$$
and 
$$
H(ru+\hat{\alpha})=H_0(u)+ F(ru)-\frac{1}{2\pi}\sum_{k=1}^m\sum_{\underset{j\neq j'}{j,j'=1}}^{N_k}\Gamma^k_j\Gamma^k_{j'}\log r .
$$
\end{proof}
\begin{Lem}\label{Lem:preliminaries_for_F_and_H_r}
The set $\cO:=\bigcup_{r\geq 0}\{r\}\times \cO_r$ is open in $[0,\infty)\times \R^{2N}$ and $H:\cO\rightarrow \R$, $(r,u)\mapsto H_r(u)$ is a $\cC^2$ function, especially $F(0)=\cH(\alpha)$. Furthermore
\begin{equation}
\begin{split}
\Gamma^k\nabla_{z^k_j}F(0)&=\Gamma^k_j\nabla_{a^k}\cH(\alpha)=0,\\
\Gamma^k\big(\nabla^2F(0)\hat{a}\big)_j^k&=\Gamma^k_j\big(\nabla^2\cH(\alpha)a\big)^k
\end{split}
\end{equation}
for any $(k,j)$ and $a\in\R^{2m}$.
\end{Lem}
\begin{proof}
Openess and smoothness are easy to check, since by (A2) indeed
$$
F(0)=\sum_{\underset{k\neq k'}{k,k'=1}}^m\Gamma^k\Gamma^{k'}G(\alpha^k,\alpha^{k'})-\sum_{k=1}^m\Gamma^k\Gamma^kg(\alpha^k,\alpha^k)=\cH(\alpha).
$$
For the derivative of $F$ with respect to $z^k_j$ we have
\begin{align*}
\nabla_{z^k_j}F(z)&=2\sum_{\underset{k'\neq k}{k'=1}}^{m}\sum_{j'=1}^{N_{k'}}\Gamma^k_j\Gamma^{k'}_{j'}\nabla_1G(z^k_j+\alpha^k,z^{k'}_{j'}+\alpha^{k'})\\
&\hspace{120pt}-2\sum_{j'=1}^{N_k}\Gamma^k_j\Gamma^{k}_{j'}\nabla_1g(z^k_j+\alpha^k,z^{k}_{j'}+\alpha^{k})
\end{align*}
and therefore
\begin{align*}
\Gamma^k\nabla_{z^k_j}F(0)&=\Gamma^k_j\Gamma^k\left(2\sum_{\underset{k'\neq k}{k'=1}}^m\Gamma^{k'}\nabla_1G(\alpha^k,\alpha^{k'})-\Gamma^k\nabla h(\alpha^k)\right)\\
&=\Gamma^k_j\nabla_{a^k}\cH(\alpha)=0
\end{align*}
by (A1$^{\hspace{1pt}\prime}$).
Now let $a\in\R^{2m}$. The $(k,j)$th component of $\nabla^2F(0)\hat{a}$ is given by
\begin{align*}
\big(\nabla^2F(0)\hat{a}\big)^k_j&= \sum_{k'=1}^m\left(\sum_{j'=1}^{N_{k'}}\nabla_{z^{k'}_{j'}}\nabla_{z^k_j}F(0)\right)a^{k'}\\
&=2\Gamma^k_j\sum_{\underset{k'\neq k}{k'=1}}^m\Gamma^{k'}\big(\nabla_1^2G(\alpha^k,\alpha^{k'})a^k+\nabla_2\nabla_1G(\alpha^k,\alpha^{k'})a^{k'}\big)\\
&\hspace{75pt}-\Gamma^k_j\Gamma^k(\underset{=\nabla^2h(\alpha^k)}{\underbrace{2\nabla^2_1g(\alpha^k,\alpha^k)+2\nabla_2\nabla_1g(\alpha^k,\alpha^k)}})a^k\\
&=\frac{\Gamma^k_j}{\Gamma^k}\sum_{k'=1}^m\nabla_{a^{k'}}\nabla_{a^k}\cH(\alpha)a^{k'}=\frac{\Gamma^k_j}{\Gamma^k}\big(\nabla^2\cH(\alpha)a\big)^k.
\end{align*}
\end{proof}
Next we turn to the functional setting. Let $\tau:=2\pi\ord(\sigma)$.
In order to find $T$-periodic solutions of \eqref{eq:Ham_omega} with $T>0$ small, we use the variational structure of \eqref{eq:Ham_r} to look for $\tau$-periodic solutions of \eqref{eq:Ham_r} with $r>0$ small. We work on the Sobolev space $H^1_\tau$ as stated in section \ref{sec:results1} and will also need the corresponding spaces $L^2_\tau$ and $H^2_\tau$.
The action functional associated to \eqref{eq:Ham_r} is given by
$$
\Phi_r(u)=\frac{1}{2}\int_0^{\tau}\ska{M_\Gamma\dot{u},J_Nu}_{\R^2}\:dt-\int_0^{\tau}H_r(u)\:dt=\Phi_0(u)-\int_0^{\tau}F(ru)\:dt+\tau\cH(\alpha).
$$
Let $\Phi:\Lambda^\prime\rightarrow\R$, $(r,u)\mapsto\Phi_r(u)$, where 
$$
\Lambda^\prime:=\set{(r,u)\in[0,\infty)\times H^1_{\tau}:(r,u(t))\in\cO\text{ for all }t\in\R}.
$$
Then $\Lambda^\prime$ is open in $[0,\infty)\times H^1_{\tau}$, since $H^1_{\tau}$ embeds into $\cC^0_{\tau}$, $\Phi\in\cC^2(\Lambda^\prime,\R)$ due to Lemma \ref{Lem:preliminaries_for_F_and_H_r} and we have to solve $\nabla\Phi_r(u)=0$ for $(r,u)\in\Lambda^\prime$ with $r>0$.

The action $\sigma*z=(\sigma_1*z^1,\ldots,\sigma_m*z^m)$ on $\R^{2N}$, induces an action on $H^1_\tau$. Let 
$$
X=\set{u\in H^1_\tau: \sigma*u(\cdot+2\pi)=u},~\Lambda=\Lambda^\prime\cap (\R\times X),~ \Lambda_r=\set{u:(r,u)\in\Lambda}.
$$
Then $X$ is a complete subspace of $H^1_\tau$ and (A3$^{\hspace{1pt}\prime}$) implies $\nabla\Phi_r(u)\in X$ for $(r,u)\in\Lambda$, since indeed $H_r(\sigma*z)=H_r(z)$, $M_\Gamma(\sigma*z)=\sigma*(M_\Gamma z)$ yield $\Phi_r(\sigma*u(\cdot+2\pi))=\Phi_r(u)$ for any $(r,u)\in\Lambda^\prime$. So it is enough to find a critical point of the restriction $\Phi_{r|\Lambda_r}:\Lambda_r\rightarrow \R$. We denote the restriction $\Phi_{|\Lambda}$ again by $\Phi$. One has
\begin{align*}
\nabla\Phi_r(u)&=\nabla\Phi_0(u)-(\id-\Delta)^{-1}r\nabla F(ru)\\
&=(\id-\Delta)^{-1}\left(-J_NM_\Gamma\dot{u}-\nabla H_0(u)-r\nabla F(ru)\right),
\end{align*}
where $\Delta:H^2_{\tau}\rightarrow L^2_{\tau}$, $u\mapsto \ddot{u}$, such that for $v\in H^1_{\tau}$, $w\in L^2_{\tau}$ the relation
$$
\ska{v,(\id-\Delta)^{-1}w}_{H^1_\tau}=\int_0^{\tau}\ska{v,w}_{\R^{2N}}\:dt=\ska{v,w}_{L^2_{\tau}}
$$
holds true.
 Note that actually $\nabla\Phi\in\cC^1(\Lambda,H^2_\tau\cap X)$, where $H^2_\tau\cap X$ is equipped with the norm $\norm{\cdot}_{H^2_\tau}$.
\section{Proof of Theorem \ref{thm:2}}\label{sec:proof_of_thm2}
For $r\rightarrow 0$ the limiting equation of \eqref{eq:Ham_r} is the decoupled system
$$
\Gamma^k_j\dot{u}^k_j=J\nabla_{u^k_j}H_{\R^2}^k(u^k_1,\ldots,u^k_{N_k}),\quad j=1,\ldots,N_k,\quad k=1,\ldots,m.
$$
So by (A3$^\prime$), $Z(t):=(Z^1(t),\ldots,Z^m(t))\in X$ is a critical point of $\Phi_0$, which of course is not isolated due to the symmetries of $H_0$. Let 
$
D=\set{\hat{a}:a\in\R^{2m}}\subset X
$ and for $\theta=(\theta_1,\ldots,\theta_m)\in (\R/\tau\Z)^m=\T^m$, $u\in X$ define the shifted version
$
\theta*u\in X$ by $\big(\theta*u\big)^k_j=u^k_j(\cdot+\theta_k).
$
Then $\Phi_0(u+\hat{a})=\Phi_0(u)=\Phi_0(\theta*u)$ for any $u\in\Lambda_0$, $\hat{a}\in D$, $\theta\in \T^m$ indeed implies that
$$
\set{\theta*Z+\hat{a}:\theta\in\T^m,a\in\R^{2m}}
$$
is a $(l+2m)$-dimensional critical manifold of $\Phi_0$. Since every $Z^k$, $k=1,\ldots,l$ is by assumption (A3$^\prime$) a $\sigma_k$-nondegenerate solution of \eqref{eq:Ham_N_k_vortex_system}, we have 
\begin{equation}\label{eq:kernel_of_Phi_0_pprime}
\ker \nabla^2\Phi_0(Z)=\spann\set{\dot{Z}^1,\ldots,\dot{Z}^l}\oplus D.
\end{equation}
Here $\dot{Z}^k$ is meant to be the element $(0,\ldots,0,\dot{Z}^k,0,\ldots,0)\in X$.
Whereas this degeneracy is natural for the limiting case $r=0$, the functionals $\Phi_r$ with $r>0$ are in general neither invariant with respect to translations by elements of $D$ nor under the action of $\T^m$ - except for synchronous time shifts $\theta=(\theta_1,\ldots,\theta_1)\in\T^m$. To deal with the degeneracy of the limiting problem we modify our equation $\nabla\Phi_r(u)=0$.

For a subspace $Y\subset X$ we denote by $P_Y:X\rightarrow Y$ the orthogonal projection onto $Y$ and by $Y^\perp$ the orthogonal complement of $Y$ in $X$. Let 
$$
\cM=\T^m*Z,\quad Y=\set{\hat{a}:a\in\ker\nabla^2\cH(\alpha)}^\perp\subset X.
$$
\begin{Lem}\label{Lem:kernel_of_auxillary_map_Psi}
There exist constants $r_0,\rho>0$ with $[0,r_0)\times B_\rho(\cM)\subset\Lambda$, such that $\psi:\cU:=[0,r_0)\times\big(B_\rho(\cM)\cap Y\big)\rightarrow Y$,
$$
\psi_r(u)=\begin{cases}
(\id-P_D)\nabla\Phi_r(u)+\frac{1}{r^2}P_{D\cap Y}\nabla\Phi_r(u),&r>0,\\
\nabla\Phi_0(u)-P_{D\cap Y}\nabla^2F(0)u,&r=0
\end{cases}
$$
is continuous, $\cC^1$ on $\cU\cap ((0,r_0)\times X)$ with $D_u\psi$ continuous up to $r=0$ and satisfies for $(r,u)\in \cU$, $r>0$:
$$
\nabla\Phi_r(u)=0 \quad \Leftrightarrow\quad\psi_r(u)=0.
$$
Moreover $\cM$ is a nondegenerate $l$-dimensional manifold of zeroes of $\psi_0$. I.e. for any $v\in\cM$ there holds
$$
\psi_0(v)=0,\quad \ker D\psi_0(v)=T_v\cM=\spann\set{\dot{v}^1,\ldots,\dot{v}^l}.
$$
\end{Lem}
\begin{proof}
As a first step observe that for positive $r$, $\bar{\psi}_r:\Lambda_r\rightarrow X$,
\begin{align}\label{eq:definition_of_psi}
\begin{split}
\bar{\psi}_r(u)&=(\id-P_D)\nabla\Phi_r(u)+\frac{1}{r^2} P_D\nabla\Phi_r(u)\\
&=\nabla\Phi_0(u)-(\id-P_D)(\id-\Delta)^{-1}r\nabla F(ru)-\frac{1}{r}P_D\nabla F(ru)
\end{split}
\end{align}
has the same zeroes as $\nabla\Phi_r$. In the second equation we used that $\nabla\Phi_0$ maps into $D^\perp$, since $\Phi_0$ is invariant with respect to translations. Clearly $\bar{\psi}$ is $\cC^1$ as long as $r>0$. Since $F$ is $\cC^2$ and $\nabla F(0)=0$, $\bar{\psi}_r$ extends as $r\rightarrow 0$ continuously to $\bar{\psi}_0:\Lambda_0\rightarrow \R$, 
$$
\bar{\psi}_0(u)=\nabla\Phi_0(u)-P_D\nabla^2F(0)u.
$$
The partial derivative $D_u\bar{\psi}:\Lambda\rightarrow\cL(X)$ is continuous as well and the regularity of $\bar{\psi}$ will carry over to $\psi$ once we have defined it.

Now let $v\in\cM$. Since $Z^k(t)=\exp\big(\pm J_{N_k}t/\ord(\sigma_k)\big)z^k$ or $Z^k(t)\equiv 0$ due to (A3$^\prime$) , we see that $\nabla^2F(0)v\in D^\perp\subset Y$. Hence $\bar{\psi}_0(v)=0$. Next
$$
\ker D\bar{\psi}_0(v)\underset{\eqref{eq:kernel_of_Phi_0_pprime}}{=}\left(\spann\set{\dot{v}^1,\ldots,\dot{v}^l}\oplus D\right)\cap\ker P_D\nabla^2F(0)
$$
and 
\begin{align*}
P_D\nabla^2F(0)\left[\sum_k\lambda_k\dot{v}^k+\hat{a}\right]=P_D\nabla^2F(0)\hat{a}.
\end{align*}
By Lemma \ref{Lem:preliminaries_for_F_and_H_r}, $\nabla^2F(0)\hat{a}=M_\Gamma \hat{b}$ with $b^k=\frac{1}{\Gamma^k}\big(\nabla^2\cH(\alpha)a\big)^k$, which projected onto $D$ gives $P_DM_\Gamma\hat{b}=\hat{c}$ with $c^k=\frac{\Gamma^k}{N_k}b^k$.
Hence we see that $\sum_k\lambda_k\dot{v}^k+\hat{a}$ is an element of the kernel of $D\bar{\psi}_0(v)$ if and only if $a\in\ker \nabla^2\cH(\alpha)$, which means $\hat{a}\in Y^\perp$. So if we restrict $\bar{\psi}$ to $\psi$ as stated in the Lemma, especially  $D\psi_0(v)=P_YD\bar{\psi}_0(v):Y\rightarrow Y$, we get
$$
\ker D\psi_0(v)=\spann\set{\dot{v}^1,\ldots,\dot{v}^l}=T_v\cM.
$$

It remains to prove that $\psi_r(u)=0$  for $r>0$ small, $u\in Y$ close to $\cM$ implies $\nabla\Phi_r(u)=0$. Note that $\psi_r(u)=0$ if and only if $P_Y\nabla\Phi_r(u)=0$.
If $\alpha\in\cF_m(\Omega)$ is a nondegenerate critical point of $\cH$ as in (A1$^\prime$)(i), we have $Y=X$ and are done. Otherwise by (A1$^\prime$), $\Omega$, $g$ and hence also $G$ and $h$ are invariant with respect to translations and/or rotations.

Assume first that (iii) of (A1$^\prime$) holds, i.e. $\lambda\nu+\Omega=\Omega$, $g(x+\lambda\nu,y+\lambda\nu)=g(x,y)$ for any $x,y\in\Omega$, $\lambda\in\R$ and some $\nu\in\R^2\setminus\{0\}$. Then $\cH(\alpha+\lambda\check{\nu})=\cH(\alpha)$, where $\check{\nu}=(\nu,\ldots,\nu)\in\R^{2m}$, and $\Phi_r(u+\lambda\hat{\check{\nu}})=\Phi_r(u)$ show that $\hat{\check{\nu}}\in Y^\perp$ and $\ska{\nabla\Phi_r(u),\hat{\check{\nu}}}=0$ for any $u\in\Lambda_r$. So if $\nu$ is the only direction, in which $g$ is invariant, then $X=Y\oplus\R\hat{\check{\nu}}$ by (A1$^\prime$) and $P_Y\nabla\Phi_r(u)=0$ automatically gives $\nabla\Phi_r(u)=0$.

If $\Omega$ and $g$ are rotational invariant, i.e. $e^{\lambda J}\Omega=\Omega$, $g(e^{\lambda J}x,e^{\lambda J}y)=g(x,y)$ for any $\lambda\in\R$, $x,y\in\Omega$, we obtain $J_m\alpha\in\ker\nabla^2\cH(\alpha)$, since $\cH(e^{\lambda J_m}\alpha)=\cH(\alpha)$ for any $\lambda\in\R$. For $\Phi_r$ there holds 
$$
\Phi_r\left(e^{\lambda J_N}\Big(u+\frac{1}{r}\hat{\alpha}\Big)-\frac{1}{r}\hat{\alpha}\right)=\Phi_r(u)
$$
and therefore $\ska{\nabla\Phi_r(u),J_N(ru+\hat{\alpha})}=0$ for any $u\in\Lambda_r$. Assuming that $\Omega$, $g$ have no other symmetry properties leads to the fact that $P_Y\nabla\Phi_r(u)=0$ implies $\nabla\Phi_r(u)=0$ as long as $X=Y\oplus \R J_N(ru+\hat{\alpha})$. Due to $J_N\hat{\alpha}\in Y^\perp$ we can find a subset $[0,r_0)\times B_\rho(\cM)\subset \Lambda$ on which this condition holds. This settles case (A1$^\prime$)(ii).

In the remaining case (A1$^\prime$)(iv), where $\Omega=\R^2$ we have to choose the neighbourhood of $\{0\}\times \cM$ such that
$$
X=Y\oplus\spann\set{\hat{\check{e}}_1,\hat{\check{e}}_2, J_N(ru+\hat{\alpha})}.
$$
\end{proof}
\begin{Rem}\label{rem:A2_violated}
If $\alpha$ is a critical point of $\cH$ not satisfying (A1$^\prime$), then Lemma \ref{Lem:kernel_of_auxillary_map_Psi} remains true with the exception that $\psi_r(u)=0$ only implies $P_Y\nabla\Phi_r(u)=0$.
\end{Rem}

So far we have reduced the degeneracy of the limiting problem by $2m=\dim D$ dimensions. To overcome the remaining degeneracy induced by the $l$ independent time shifts of $Z^1,\ldots,Z^l$ we perform a Lyapunov-Schmidt reduction.

For $v\in\cM$ denote by $P_v:X\rightarrow T_v\cM\subset Y$ the orthogonal projection onto $T_v\cM$. Moreover define $\tilde{\psi}:\tilde{\cU}:=[0,r_0)\times\cM\times (B_\rho(0)\cap Y)\rightarrow Y$,
$$
\tilde{\psi}(r,v,w)=(\id-P_v)\psi_r(v+w)+P_v w.
$$
Since $\cM\ni v\mapsto P_v\in\cL(X)$ is $\cC^1$, we have $\tilde{\psi}\in\cC^1$ where $r>0$, as well as continuity of $\tilde{\psi}$, $D_v\tilde{\psi}$, $D_w\tilde{\psi}$ on all of $\tilde{\cU}$. For $(r,v,w)\in\tilde{\cU}$ there holds
$$
\psi_r(v+w)=0,~w\perp T_v\cM
 \quad\Longleftrightarrow \quad\begin{cases}
P_v\psi_r(v+w)=0,\\
\tilde{\psi}(r,v,w)=0.
\end{cases}
$$
\begin{Lem}\label{lem:ift}
Shrinking $r_0>0$ and $\rho>0$ if necessary, we find a continuous map $W:[0,r_0)\times \cM\rightarrow B_\rho(0)\cap Y$ satisfying $W(r,v)\perp T_v\cM$ for any $(r,v)\in [0,r_0)\times\cM$ and 
$$
\tilde{\psi}(r,v,w)=0\quad \Longleftrightarrow\quad w=W(r,v)
$$
on $\tilde{\cU}$. Moreover each $W(r,\cdot):\cM\rightarrow B_\rho(0)$ is equivariant with respect to the orthogonal action of $\set{\theta\in\T^m:\theta_1=\ldots=\theta_m}\cong S^1$ on $X$. Concerning regularity we have $W\in\cC^1((0,r_0)\times \cM)$, and $D_vW$ is as $W$ itself continuous up to $r=0$.  
\end{Lem}
\begin{proof}
Let $v\in\cM$. One has $\tilde{\psi}(0,v,0)=0$ and 
$$
T:=D_w\tilde{\psi}(0,v,0)=(\id-P_{v})D\psi_0(v)+P_{v}=D\psi_0(v)+P_{v}
$$ 
has trivial kernel by Lemma \ref{Lem:kernel_of_auxillary_map_Psi}. But note that $\range (T)\neq Y$, in fact $T$ is an isomorphism between $Y$ and $H^2_\tau\cap Y$, as can be seen in the following way:

Let $P_0:H^1_{\tau}\rightarrow \R^{2N}$ be the orthogonal projection onto the space of constant functions and $L:H^s_{\tau}\rightarrow H^{s+1}_{\tau}$, $u\mapsto (\id-\Delta)^{-1}(-J_NM_\Gamma\dot{u})+P_0u$. Then $L$ is an isomorphism, also when viewed as a mapping from $Y\rightarrow H^2_\tau\cap Y$. Since $v$ is smooth, $L^{-1}\tilde{\psi}(0,v,\cdot)-\id:B_\rho(0)\cap Y\rightarrow Y$ is continuously differentiable and maps bounded subsets onto relatively compact subsets. Hence $L^{-1}T:Y\rightarrow Y$ is an index $0$ Fredholm operator with trivial kernel. Therefore $T:Y\rightarrow L Y=H^2_{\tau}\cap Y$ is an isomorphism.

Note also that $\tilde{\psi}$ viewed as a map into $H^2_{\tau}\cap Y$ with $\norm{\cdot}_{H^2_\tau}$ instead of $Y$ has the same regularity as the original $\tilde{\psi}$. So the implicit function theorem yields local maps $W_v$ solving the stated equation on $[0,r_v)\times U_v \times B_{\rho_v}(0)$, where $U_v\subset \cM$ is an open neighbourhood of $v$.

However the compactness of $\cM$ and the uniqueness of the solution allow us to construct a global map $W$ as requested by the Lemma. The equivariance with respect to synchronuous time shifts follows from the corresponding equivariance of $\tilde{\psi}$, i.e. $\tilde{\psi}(r,\theta*v,\theta*w)=\theta*\tilde{\psi}(r,v,w)$.
\end{proof}
For $r\in(0,r_0)$, $v\in\cM$ it now remains to solve 
$$
P_v\psi_r(v+W(r,v))=0.
$$
Therefore let $\varphi:[0,r_0)\times\cM\rightarrow \R$,
\begin{equation}\label{eq:definition_of_varphi}
\varphi(r,v)=\varphi_r(v)=\Phi_r(v+W(r,v)).
\end{equation}
\begin{Lem}\label{lem:critical_points_of_varphi_imply_}
There exists $r_1\in(0,r_0)$ such that $r\in(0,r_1)$, $D\varphi_r(v)=0$ implies $P_v\psi_r(v+W(r,v))=0$.
\end{Lem}
\begin{proof}
Differentiation of $P_vW(0,v)=0$ shows that $P_vD_vW(0,v)=0$ and therefore $P_vD_vW(r,v)=o(1)$ uniformly in $v\in\cM$ as $r\rightarrow 0$. Choose $r_1\in(0,r_0)$ such that $\norm{P_vD_vW(r,v)}_{\cL(T_v\cM)}\leq \frac{1}{2}$ for every $(r,v)\in(0,r_1)\times\cM$.

Assume $D\varphi_r(v)=0$ for some $0<r<r_1$, $v\in\cM$. Using $P_v\circ P_D=0$ one sees that $\tilde{\psi}(r,v,W(r,v))=0$ implies 
\begin{equation}\label{eq:nabla_Phi_gleich_id_Pv_nabla_Phi}
(\id-P_v)P_Y\nabla \Phi_r(v+W(r,v))=0.
\end{equation}
Thus we obtain for $v'\in T_v\cM$
\begin{align}\label{eq:derivative_of_varphi}
\begin{split}
0=D\varphi_r(v)v'&=\ska{\nabla\Phi_r(v+W(r,v)),(\id+D_vW(r,v))v'}\\
&=\ska{P_Y\nabla\Phi_r(v+W(r,v)),P_Y(\id+D_vW(r,v))v'}\\
&=\ska{P_vP_Y\nabla\Phi_r(v+W(r,v)),(\id+P_vD_vW(r,v))v'}
\end{split}
\end{align}
and conclude $P_v\psi_r(v+W(r,v))=P_vP_Y\nabla\Phi_r(v+W(r,v))=0$, since the map $\id+P_vD_vW(r,v):T_v\cM\rightarrow T_v\cM$ is an isomorphism.
\end{proof}
Now it remains to investigate critical points of $\varphi_r$ for $r\in(0,r_1)$.
\begin{proof}[Proof of Theorem \ref{thm:2}]
Let $r\in(0,r_1)$. The reduced functional $\varphi_r$ is invariant with respect to the action of $\set{\theta\in\T^m:\theta_1=\ldots=\theta_m}$, which is smooth on $\cM$. So every critical point of $\varphi_r$ belongs to a whole orbit of critical points. If $l=1$, we are done. Otherwise we can find on each of the critical orbits a point of the form $(v^1,\ldots,v^{l-1},Z^l,0,\ldots,0)\in\cM$. Therefore the number of critical orbits is given by the number of critical points of $\T^{l-1}\rightarrow \R$,
$
\theta\mapsto\varphi_r\big((\theta_1,\ldots,\theta_{l-1},0,\ldots,0)*Z\big),
$
for which the Lusternik-Schnirelmann category of $\T^{l-1}$ provides $l$ as a minimal bound, see for example \cite{cornea_lusternik-schnirelmann_2003}.

This way we have found for every $r\in(0,r_1)$ $l$ distinct critical points of $\Phi_r$. Let $u=v+W(r,v)\in Y$ be one of them. Then $z(t)=ru(t/r^2)+\hat{\alpha}$ is by construction a $T(r)=\tau r^2=2\pi\ord(\sigma)r^2$-periodic solution of \eqref{eq:Ham_omega}, for which the properties of Theorem \ref{thm:2} hold.
\end{proof}
\section{\texorpdfstring{Additional information and the case $l=1$}{Additional information and the case l=1}}\label{sec:the_case_l_1}
For now we just continue our investigation with $l\in\set{1,\ldots,m}$ arbitrary. Higher order derivatives with respect to $z$ are written as $F'''$, $F^{(4)}$ and so on.
\begin{Lem}\label{lem:regularity_of_W}
Let $g\in\cC^k(\Omega\times\Omega,\R)$ with $k\geq 2$. The map $W:[0,r_0)\times\cM\rightarrow H^1_\tau$ is of class $\cC^{k-2}$. Furthermore if $k\geq 3$, $\partial_rW(0,v)\in D$ for any $v\in \cM$. 
\end{Lem}
\begin{proof} Since $\cM\ni v\mapsto P_v\in\cL(X)$ is $\cC^\infty$ and since $W$ is implicitly defined, the regularity of $W$ is induced by $\psi$. With $g\in\cC^k$ we also have $F\in\cC^k$ and hence $\Phi\in\cC^k$. Then by the definition of $\psi$ in \ref{Lem:kernel_of_auxillary_map_Psi} one sees that $\psi$ is indeed of class $\cC^{k-2}$ provided
$\kappa:\cU\rightarrow L^2(\R/\tau\Z,\R^{2N})$,
$$
\kappa(r,u)=\begin{cases}
\frac{1}{r}\nabla F(ru),&r>0,\\
\nabla^2F(0)u,&r=0
\end{cases}
$$
is $\cC^{k-2}$. In order to proove this observe that $\kappa$ is $\cC^k$ as long as $r>0$. The continuity up to $r=0$ follows as in the proof of Lemma \ref{Lem:kernel_of_auxillary_map_Psi} from the fact that $F$ is $\cC^2$ and that $\nabla F(0)=0$. Also the partial dervivatives that include at least one differentiation of $\kappa$ with respect to $u$ are easily seen to extend in a continuous way as $r\rightarrow 0$. So we have to look at the partial derivative 
\begin{align*}
\partial^{k-2}_r\kappa(r,u)=\sum_{j=0}^{k-2}\frac{(k-2)!}{j!}(-1)^{k-j}\frac{1}{r^{k-1-j}}F^{(j+1)}(ru)[u]^j,
\end{align*}
where $(r,u)\in\cU$ with $r>0$. Now a (pointwise) expansion of $F^{(j+1)}$ gives
$$
F^{(j+1)}(ru)[u]^j=\sum_{l=0}^{k-2-j}\frac{r^l}{l!}F^{(j+1+l)}(0)[u]^{j+l}+\frac{r^{k-1-j}}{(k-1-j)!}F^{(k)}(\xi u)[u]^{k-1}
$$ 
for some $\xi=\xi(j,u,t)\in(0,r)$. But as $r\rightarrow 0$ we obtain for the remainder 
$$
F^{(k)}(\xi u)[u]^{k-1}=F^{(k)}(0)[u]^{k-1}+o(1)
$$ with respect to $\norm{\cdot}_{L^2_\tau}$ and uniformly in $u\in B_\rho(\cM)$. Thus
\begin{align*}
\partial^{k-2}_r\kappa(r,u)&=\sum_{j=0}^{k-2}\sum_{l=0}^{k-2-j}\frac{(k-2)!(-1)^{k-j}}{j!l!}\frac{1}{r^{k-1-l-j}}F^{(j+1+l)}(0)[u]^{j+l}\\
&\hspace{40pt}+\sum_{j=0}^{k-2}\frac{(k-2)!(-1)^{k-j}}{j!(k-1-j)!}F^{(k)}(0)[u]^{k-1}+o(1)
\\&=\sum_{n=0}^{k-2}\left(\frac{(k-2)!(-1)^k}{n!r^{k-1-n}}F^{(n+1)}(0)[u]^n\sum_{j=0}^n\frac{n!(-1)^j}{j!(n-j)!}\right)
\\&\hspace{40pt}+F^{(k)}(0)[u]^{k-1}\int_0^1(1-s)^{k-2}\:ds +o(1)
\\&=\frac{1}{k-1}F^{(k)}(0)[u]^{k-1}+o(1).
\end{align*}
So the partial derivatives $\partial_r^{j}\kappa$, $j=1,\ldots,k-2$ exist and are continuous on all of $\cU$.

For the second part assume that $g\in\cC^3$. Now $W$ is $\cC^1$ on all of $[0,r_0)\times\cM$ and we know by Lemma \ref{lem:ift} that
$$
(\id-P_v)P_Y\nabla\Phi_r(v+W(r,v))=0,\quad P_vW(r,v)=0
$$
for $r>0$ small, cf. equation \eqref{eq:nabla_Phi_gleich_id_Pv_nabla_Phi}.
Differentiation of both equations with respect to $r$ at $r=0$ and the use of $\partial_r\nabla\Phi_0(v)=0$ as well as $(\id-P_v)P_Y\nabla^2\Phi_0(v)=\nabla^2\Phi_0(v)$ shows 
$$
\partial_rW(0,v)\in \ker\nabla^2\Phi_0(v)\cap (T_v\cM)^\perp=D.
$$
\end{proof}
\begin{proof}[Proof of Theorem \ref{thm:case_m1}]
Let now $l=1$. In that case the reduced map $\varphi_r$ is in fact constant. Hence the demanded solutions of $\nabla\Phi_r(u)=0$ can be parameterized by $u:[0,r_1)\rightarrow H^1_\tau$, $r\mapsto u^{(r)}=Z+W(r,Z)$, where $r_1>0$ is taken from Lemma \ref{lem:critical_points_of_varphi_imply_} and $Z=(Z^1,0\ldots,0)\in\cM$. By \ref{lem:regularity_of_W} this parametrization is indeed $\cC^{k-2}$ provided $g\in\cC^k$, $k\geq 2$ and $\partial_ru^{(0)}\in D$ when $k\geq 3$.
\end{proof}
\begin{Rem}\label{rem:A4}
For the case $l>1$ a corresponding result would be true provided one knows that $\varphi_r$ for every $r>0$ small is a Morse function. This would not only imply that the solution set of $\nabla\Phi_r(u)=0$ close to $\{0\}\times\cM$ is a union of graphs but also increase for fixed $r>0$ the number of existing solutions to $2^{l-1}$, which is the bound given by Morse theory. A fourth-order expansion of $\varphi_r$ in $\cC^2(\cM,\R)$, which we don't carry out in detail, shows that the improvements would hold provided
 $f:\cM\rightarrow\R$,
$$
f(v)=\int_0^\tau F^{(4)}(0)[v]^4+6\ska{\nabla^2F(0)v,\partial_r^2W(0,v)}_{\R^{2N}}\:dt
$$
has up to synchronous time shifts only nondegenerate critical points. But this condition has so far not been verified for specific examples.
\end{Rem}

\section{An explicit example}\label{sec:example}
With Examples \ref{ex:relative_equilibria} and \ref{ex:relative_equilibrium_2} we have already seen some relative equilibrium solutions that are $\sigma$-nondegenerate or just nondegenerate and therefore can be choosen in (A3$^\prime$) for theorem \ref{thm:2}. Independent of the relative equilibrium solutions we also need for (A1$^\prime$) a nondegenerate or not too degenerate critical point of the $m$-vortex Hamiltonian $\cH$. We will verify this for Example \ref{Ex:unit_disc}. I.e. we look at the $2$-vortex system in the unit disc $\Omega=B_1(0)$ with vorticities $\Gamma^1=1$, $\Gamma^2=-1$. By combining for example a Thomson $N_1$-Gon configuration with vorticities $\Gamma^1_j=\frac{1}{N_1}$, $j=1,\ldots,N_1$ and a collinear configuration of $N_2$ vortices of strengths $\Gamma^2_j=-\frac{1}{N_2}$, $j=1,\ldots,N_2$ or another Thomson configuration we obtain therefore periodic solutions of \eqref{eq:Ham_omega} in the unit disc for an arbitrary number of $N=N_1+N_2\geq 3$ vortices that are not rigidly rotating around the center of the disc.

The regular part of the Dirichlet Green's function in $B_1(0)$ is given by
$$
g(x,y)=g_{B_1(0)}(x,y)=-\frac{1}{4\pi}\log\left(\abs{x}^2\abs{y}^2-2\ska{x,y}_{\R^2}+1\right)
$$
and 
$$
h(x)=h_{B_1(0)}(x)=-\frac{1}{2\pi}\log(1-\abs{x}^2),
$$
such that the Hamiltonian defined on $\cF_2(B_1(0))$ is given by
\begin{align*}
\cH(a^1,a^2)&=\frac{1}{\pi}\left(\log\abs{a^1-a^2}-\frac{1}{2}\log\left(\abs{a^1}^2\abs{a^2}^2-2\ska{a^1,a^2}_{\R^2}+1\right)\right)\\
&\hspace{70pt}+\frac{1}{2\pi}\left(\log\big(1-\abs{a^1}^2\big)+\log\big(1-\abs{a^2}^2\big)\right).
\end{align*}
Let $R(y)=\frac{y}{\abs{y}^2}$ be the reflection at the unit circle, then
\begin{align*}
\pi\nabla_1\cH(a^1,a^2)&=\frac{a^1-a^2}{\abs{a^1-a^2}^2}-\frac{a^1-R(a^2)}{\abs{a^1-R(a^2)}^2}-\frac{a^1}{1-\abs{a^1}^2},\\
\pi\nabla_2\cH(a^1,a^2)&=\frac{a^2-a^1}{\abs{a^2-a^1}^2}-\frac{a^2-R(a^1)}{\abs{a^2-R(a^1)}^2}-\frac{a^2}{1-\abs{a^2}^2}.
\end{align*}
The ansatz $\alpha^1=(\mu,0)$, $\alpha^2=(-\mu,0)$ with $\mu>0$ shows that $\alpha=(\alpha^1,\alpha^2)$ is a critical point of $\cH$ if and only if
\begin{equation}\label{eq:relation_of_mu}
\mu^4=1-4\mu^2,
\end{equation}
which means $\mu=\sqrt{\sqrt{5}-2}$. For the second derivatives at the critical point $\alpha=(\mu,0,-\mu,0)$ we get with a repeated use of \eqref{eq:relation_of_mu}
\begin{align*}
\pi\nabla_1^2\cH(\alpha)&=\left(\frac{1}{4\mu^2}-\frac{1}{(\mu+\frac{1}{\mu})^2}\right)\begin{pmatrix}
-1 & 0 \\
0 & 1
\end{pmatrix}-\frac{1}{(1-\mu)^2}\begin{pmatrix}
1+\mu^2 & 0 \\
0 & 1-\mu^2
\end{pmatrix}\\
&=\frac{1}{26\mu^2-6}\begin{pmatrix}
-6\mu^2+1 & 0 \\
0 & 4\mu^2-1
\end{pmatrix},
\end{align*}
\begin{align*}
\pi\nabla_2\nabla_1\cH(\alpha)&=\frac{1}{4\mu^2}\begin{pmatrix}
-1 & 0 \\
0 & 1
\end{pmatrix}+\frac{1}{(1+\mu^2)^2}\begin{pmatrix}
1 & 0 \\
0 & 1
\end{pmatrix},\\
&=\frac{1}{20\mu^2-4}\begin{pmatrix}
\mu^2+1 & 0\\
0 &3\mu^2-1
\end{pmatrix}
\end{align*}
and $\nabla^2_2\cH(\alpha)=\nabla^2_1\cH(\alpha)$, $\nabla_1\nabla_2\cH(\alpha)=\nabla_2\nabla_1\cH(\alpha)$. So the Hessian of $\cH$ is given by
$$
\pi\nabla^2\cH(\alpha)=\begin{pmatrix}
\frac{-6\mu^2+1}{26\mu^2-6} & 0 & \frac{\mu^2+1}{20\mu^2-4} & 0 \\
0 & \frac{4\mu^2-1}{26\mu^2-6} & 0 & \frac{3\mu^2-1}{20\mu^2-4}\\
 \frac{\mu^2+1}{20\mu^2-4} & 0 &\frac{-6\mu^2+1}{26\mu^2-6} & 0\\
0& \frac{3\mu^2-1}{20\mu^2-4}& 0 & \frac{4\mu^2-1}{26\mu^2-6}  
 \end{pmatrix}.
$$
Using \eqref{eq:relation_of_mu} one can verify that the second and the fourth column are identical. This corresponds to the degeneracy induced by the rotational invariance, which means $J_2\alpha=(0,-\mu,0,\mu)\in\ker\nabla^2\cH(\alpha)$. On the other hand one easily sees that the first three columns are linearly independent. This shows that $\alpha$ is a critical point of the $2$ vortex Hamiltonian $\cH$ satisfying condition (A2$^\prime$)(ii) as it has been stated in Example \ref{Ex:unit_disc}.

\vspace{20pt}
\textbf{Acknowledgements.} I would like to thank Prof. Thomas Bartsch not only for his helpful comments and questions regarding this article but also for his whole support during my PhD studies in Gie\ss en.

\printbibliography
\vspace{40pt}
\noindent Bj\"orn Gebhard\\
 Mathematisches Institut\\
 Universit\"at Gie\ss en\\
 Arndtstr.\ 2\\
 35392 Gie\ss en\\
 Germany\\
 Bjoern.Gebhard@math.uni-giessen.de
\end{document}